\documentclass[12pt,a4paper,twoside]{amsart}
\usepackage[latin1]{inputenc}
\usepackage{eucal}
\usepackage{amssymb}
\usepackage{amsmath}
\usepackage{amsthm}

\DeclareMathOperator{\rank}{rank}

\DeclareMathOperator{\vol}{vol}

\title{Sums and differences of four $k$-th powers}
\author{Oscar Marmon}
\date{}
\address{Mathematical Sciences \\ Chalmers University of Technology \\ SE-412 96 Gothenburg}
\address{Mathematical Sciences \\ University of Gothenburg \\ SE-412 96 Goth{-}enburg}
\email{marmon@chalmers.se}

\begin{document}

\newcommand{\epsi}{\varepsilon}
\newcommand{\xx}{\mathbf{x}}
\newcommand{\xxi}{\boldsymbol{\xi}}
\newcommand{\eeta}{\boldsymbol{\eta}}
\newcommand{\0}{\boldsymbol{0}}
\newcommand{\bb}{\mathbf{b}}
\newcommand{\uu}{\mathbf{u}}
\newcommand{\ee}{\mathbf{e}}
\newcommand{\vv}{\mathbf{v}}
\newcommand{\yy}{\mathbf{y}}
\newcommand{\zz}{\mathbf{z}}
\newcommand{\ww}{\mathbf{w}}
\newcommand{\ff}{\mathbf{f}}
\newcommand{\cc}{\mathbf{c}}
\newcommand{\dd}{\mathbf{d}}
\newcommand{\hh}{\mathbf{h}}
\newcommand{\ttt}{\mathbf{t}}
\renewcommand{\aa}{\mathbf{a}}
\newcommand{\bK}{\mathbf{K}}
\newcommand{\bB}{\mathbf{B}}
\newcommand{\ZZ}{\mathbb{Z}}
\newcommand{\Zpol}{\ZZ[X_1,\ldots,X_n]}
\newcommand{\FF}{\mathbb{F}}
\newcommand{\RR}{\mathbb{R}}
\newcommand{\NN}{\mathbb{N}}
\newcommand{\CC}{\mathbb{C}}
\renewcommand{\AA}{\mathbb{A}}
\newcommand{\PP}{\mathbb{P}}
\newcommand{\GG}{\mathbb{G}}
\newcommand{\QQ}{\mathbb{Q}}
\newcommand{\sB}{\mathsf{B}}
\newcommand{\cP}{\mathcal{P}}
\newcommand{\cR}{\mathcal{R}}
\newcommand{\cB}{\mathcal{B}}
\newcommand{\cE}{\mathcal{E}}
\newcommand{\cC}{\mathcal{C}}
\newcommand{\cN}{\mathcal{N}}
\newcommand{\cM}{\mathcal{M}}
\newcommand{\cL}{\mathcal{L}}
\newcommand{\cD}{\mathcal{D}}
\newcommand{\cA}{\mathcal{A}}
\newcommand{\cK}{\mathcal{K}}
\newcommand{\cF}{\mathcal{F}}
\newcommand{\cZ}{\mathcal{Z}}
\newcommand{\cX}{\mathcal{X}}
\newcommand{\cG}{\mathcal{G}}
\newcommand{\cH}{\mathcal{H}}
\newcommand{\cI}{\mathcal{I}}
\newcommand{\cO}{\mathcal{O}}
\newcommand{\cW}{\mathcal{W}}
\newcommand{\ud}{\mathrm{ud}}
\newcommand{\Zar}{\mathrm{Zar}}
\newcommand{\fm}{\mathfrak{m}}
\newcommand{\fp}{\mathfrak{p}}
\newcommand{\smod}[1]{\,(#1)}
\newcommand{\spmod}[1]{\,(\mathrm{mod}\,{#1})}

\newtheorem{thm}{Theorem}[section]
\newtheorem{lemma}[thm]{Lemma}
\newtheorem*{lemma*}{Lemma}
\newtheorem{prop}[thm]{Proposition}
\newtheorem*{prop*}{Proposition}
\newtheorem*{thm*}{Theorem}
\newtheorem{claim}[thm]{Claim}
\newtheorem{cor}[thm]{Corollary}
\newtheorem*{conj*}{Conjecture}
\newtheorem{conj}{Conjecture}
\newtheorem*{sats*}{Sats}
\theoremstyle{remark}
\newtheorem*{note*}{Note}
\newtheorem{note}{Note}
\newtheorem*{rem*}{Remark}
\newtheorem{rem}[thm]{Remark}
\newtheorem{example}[thm]{Example}
\newtheorem*{acknowledgement*}{Acknowledgement}
\newtheorem*{question*}{Question}
\newtheorem*{answer*}{Answer}
\theoremstyle{definition}
\newtheorem*{def*}{Definition}
\newtheorem{notation}[thm]{Notation}
\newtheorem*{notation*}{Notation}

%%%%%%%%%%%%%%%%%%%%%%%%%%%%%%%%%%%%%%%%%%%%%%%%%%%%%%%%%%%%%%%%%%%%%%%%%%%%%%%%%%%%%%%%%%%%%%%%%%%%%%%%%%%%%
% Abstract here for amsart
\begin{abstract}
We prove an upper bound for the number of representations of a positive integer $N$ as the sum of four $k$-th powers of integers of size at most $B$, using a new version of the determinant method developed by Heath-Brown, along with recent results by Salberger on the density of integral points on affine surfaces.

More generally we consider representations by any integral diagonal form. The upper bound has the form $O_{N}(B^{c/\sqrt{k}})$, whereas earlier versions of the determinant method would produce an exponent for $B$ of order $k^{-1/3}$ (uniformly in $N$) in this case.

Furthermore, we prove that the number of representations of a positive integer $N$ as a sum of four $k$-th powers of non-negative integers is at most $O_{\varepsilon}(N^{1/k+2/k^{3/2}+\varepsilon})$ for $k \geq 3$, improving upon bounds by Wisdom.
\end{abstract}

%%%%%%%%%%%%%%%%%%%%%%%%%%%%%%%%%%%%%%%%%%%%%%%%%%%%%%%%%%%%%%%%%%%%%%%%%%%%%%%%%%%%%%%%%%%%%%%%%%%%%%%%%%

\maketitle
\let\languagename\relax

\hyphenation{Chal-mers}

%%%%%%%%%%%%%%%%%%%%%%%%%%%%%%%%%%%%%%%%%%%%%%%%%%%%%%%%%%%%%%%%%%%%%%%%%%%%%%%%%%%%%%%%%%%%%%%%%%%%%%%
% Abstract here for article
% \begin{abstract}
% We prove an upper bound for the number of representations of a positive integer $N$ as the sum of four $k$-th powers of integers of size at most $B$, using a new version of the determinant method developed by Heath-Brown, along with recent results by Salberger on the density of integral points on affine surfaces.
% 
% More generally we consider representations by any integral diagonal form. The upper bound has the form $O_{N}(B^{c/\sqrt{k}})$, whereas earlier versions of the determinant method would produce an exponent for $B$ of order $k^{-1/3}$ (uniformly in $N$) in this case.
% 
% Furthermore, we prove that the number of representations of a positive integer $N$ as a sum of four $k$-th powers of non-negative integers is at most $O_{\varepsilon}(N^{1/k+2/k^{3/2}+\varepsilon})$ for $k \geq 3$, improving upon bounds by Wisdom.
% \end{abstract}
%%%%%%%%%%%%%%%%%%%%%%%%%%%%%%%%%%%%%%%%%%%%%%%%%%%%%%%%%%%%%%%%%%%%%%%%%%%%%%%%%%%%%%%%%%%%%%%%%%%%%%%

\section{Introduction}

In this paper, we shall study the number of representations of a positive integer $N$ using four $k$-th powers. We consider two different versions of this problem. The main part of the paper concerns solutions to the equation
\begin{equation}
x_1^k \pm x_2^k \pm x_3^k \pm x_4^k = N
\end{equation}
in integers $x_i$, positive or negative. Our treatment of this problem is inspired by a recent paper of Heath-Brown \cite{Heath-Brown09} where he studies the equation
\begin{equation}
\label{eq:threepowers}
x_1^k \pm x_2^k \pm x_3^k = N.
\end{equation}

More precisely, he estimates the number of integral solutions to \eqref{eq:threepowers}, with $\max |x_i| \leq B$, that are not trivial in the sense that $\pm x_i^k = N$ for some $i$. Assuming that $N \ll B$, Heath-Brown proves that there are $O_k(B^{10/k})$ such solutions for $k \geq 3$.

The method used by Heath-Brown is a new approach to the determinant method of Bombieri and Pila \cite{Bombieri-Pila}. Rather than counting integral points on the affine surface defined by \eqref{eq:threepowers}, an approach that would yield an exponent of order $1/\sqrt{k}$ (using the version of the determinant method developed in \cite{Heath-Brown02} and \cite{Heath-Brown06}), he studies rational points near the projective curve given by $x_1^k \pm x_2^k \pm x_3^k = 0$.

Our aim in this paper is to study the corresponding problem in four variables, using the approach of \cite{Heath-Brown09}. The method works for arbitrary non-singular diagonal forms, so we state our main result in that generality. Let $\aa=(a_1,a_2,a_3,a_4)$ be a quadruple of non-zero integers, $k\geq 3$ an integer, and $N$ a positive integer. Let $\cR(N,B)$ be the number of quadruples $(x_1,x_2,x_3,x_4) \in \ZZ^4$ satisfying
\begin{equation}
\label{eq:fourpowers}
a_1 x_1^k + a_2 x_2^k + a_3 x_3^k + a_4 x_4^k = N
\end{equation}
and $\max |x_i| \leq B$. We note that the trivial estimate $\cR(N,B) = O(B^{2+\epsi})$ can be deduced easily using known results for Thue equations (see Proposition \ref{prop:thue} below).

We call a solution $\xx$ to \eqref{eq:fourpowers} \emph{special} if either $a_i x_i^k = N$ for some index $i$ or  $a_i x_i^k + a_j x_j^k = N$ for some pair of indices $i,j$.
If $X \subset \AA^4$ denotes the hypersurface defined by \eqref{eq:fourpowers}, then these solutions are all contained in a proper subvariety of $X$, namely the union of all lines on $X$ (see Section \ref{sec:fermat}). We shall see (Proposition \ref{prop:special}) that the special solutions contribute at most $O_{k,\epsi}(B N^\epsi)$ to $\cR(N,B)$. Thus, let $\cR_0(N,B)$ be the number of non-special solutions to \eqref{eq:fourpowers} satisfying $\max \vert x_i \vert \leq B$. Then we shall prove the following estimate.

%%%%%%%%%%%% Gammal uppskattning $9/\sqrt{2k}$ %%%%%%%%%%

%
% \begin{thm}
% \label{thm:main}
% \[
% \cR_0(N,B) \ll_\epsi B^{9/\sqrt{2k}+\epsi}(B^{2/\sqrt{k}} + B^{1/\sqrt{k}+6/(k+3)}).
% \]
% In particular,
% \[
% \cR_0(N,B) \ll_\epsi B^{9/\sqrt{2k}+2/\sqrt{k}+\epsi}
% \]
% for $k \geq 30$, and $\cR(N,B) \ll_\epsi B^{1+\epsi}$ (that is, the special solutions dominate) for $k \geq 70$.

%%%%%%%%%%%% Uppdelning i färre fall %%%%%%%%%

% \[
% \cR(N,B) \ll_\epsi \begin{cases}
% B^{\psi(k) + 1/\sqrt{k} + 6/(k+3) + \epsi} + B^{\psi(k) + 2/\sqrt{k} + \epsi} + B^{1+\epsi}, & k \geq 22, \\
% %B^{1 + 1/\sqrt{k} + 6/(k+3) + \epsi} + B^{1 + 2/\sqrt{k} + \epsi}, & 22 \leq k \leq 40, \\
% B^{1 + 1/\sqrt{k} + 1/4 + \epsi} + B^{1 + 2/\sqrt{k} + \epsi}, & \hspace{-2ex} 6 \leq k \leq 21, \\
% B^{3/2 + 1/\sqrt{5}+ \epsi}, & k = 5,
% \end{cases}
% \]
% where $\psi(k) = \min \{1,9/\sqrt{2k}\}$.

%%%%%%%%%%%% Uppdelning i många fall. %%%%%%%%%%%

% \[
% \cR(N,B) \ll_\epsi \begin{cases}
% B^{3/2+1/\sqrt{5}+\epsi}, & k = 5, \\
% B^{1+2/\sqrt{k}+\epsi}, & 6 \leq k \leq 16, \\
% B^{5/4+1/\sqrt{k}+\epsi}, & 17 \leq k \leq 21, \\
% B^{1+2/\sqrt{k}+\epsi}, & 22 \leq k \leq 29, \\
% B^{1+1/\sqrt{k}+6/(k+3)+\epsi}, & 30 \leq k \leq 40, \\
% B^{(9+2\sqrt{2})/\sqrt{2k}+\epsi}, & 41 \leq k \leq 69, \\
% B^{1+\epsi}, & k \geq 70.
% \end{cases}
% \]

\begin{thm}
\label{thm:main}
For any $\epsi > 0$ we have
\begin{equation}
\label{eq:main}
\cR_0(N,B) \ll_{\aa,N,\epsi} B^{16/(3\sqrt{3k})+\epsi}(B^{2/\sqrt{k}} + B^{1/\sqrt{k}+6/(k+3)}).
\end{equation}
In particular,
%\[
%\cR_0(N,B) \ll_{N,\epsi} B^{16/(3\sqrt{3k})+2/\sqrt{k}+\epsi}
%\]
%for $k \geq 30$, and
$\cR(N,B) \ll_{\aa,N} B$
%(that is, the special solutions dominate)
for $k \geq 27$.
\end{thm}

\begin{rem*}
The exponent $16/(3\sqrt{3k})$ in Theorem \ref{thm:main} is to be compared with the exponent $3/k^{1/3}$ that could be obtained (uniformly in $N$) by applying the ``ordinary'' determinant method of Heath-Brown \cite[Thm. 15]{Heath-Brown06} in this case. Furthermore, we remark that the bound \eqref{eq:main} is non-trivial for $k \geq 8$.
\end{rem*}

% \begin{rem}
% The exponent $9/\sqrt{2k}$ in Theorem \ref{thm:main} is to be compared with the exponent $3/k^{1/3}$ that could be obtained by applying the 'ordinary' determinant method of Heath-Brown \cite[Thm. 15]{Heath-Brown06} in this case. Thus, Theorem \ref{thm:main} is of most interest when $k$ is quite large. It should be noted, however, that there is much room for improvement in the theorem for small $k$ (see Section \ref{sec:auxiliary}).
% \end{rem}
%
% \begin{rem}
% Although we have stated the theorem in the restrictive case where $N=O(1)$, to get as sharp an exponent as possible, the same method easily yields similar results for, say, $N=O(B^{\alpha k})$ where $0 \leq \alpha < 1$.
% \end{rem}

The estimate in Theorem \ref{thm:main} is proven by combining the ideas from \cite{Heath-Brown09} with recent results by Salberger \cite{Salberger09} about the density of integral points on affine surfaces.

In Sections \ref{sec:parameterization} and \ref{sec:auxiliary} we adapt Heath-Brown's arguments to the four-variable case. As with other instances of the determinant method, the output is a number of auxiliary forms, allowing us to estimate $\cR(N,B)$ through counting integral points of bounded height on a number of affine algebraic surfaces. In doing this, we use results by Salberger, discussed in Section \ref{sec:fermat}, concerning the geometry of Fermat hypersurfaces. The proof of Theorem \ref{thm:main} is finished in Section \ref{sec:integralpoints}.

It is implicit in Theorem \ref{thm:main} that $N$ is fixed and small. If $N$ is allowed to grow as $B \to \infty$, we have the following more precise estimate.

\begin{thm}
\label{thm:mainbigN}
Suppose that $N = O(B^{k-\tau})$, where $4/3 < \tau < k$. Then we have
% \begin{equation}
% \label{eq:mainbigN}
% \cR_0(N,B) \ll_{\epsi} B^{\frac{16}{3\sqrt{3k}} + \epsi} N^{\frac{8}{\sqrt{3}k^{3/2}}\left(\frac{\sqrt{2}}{(1-\mu)}-\frac{1}{3}\right)}(B^{2/\sqrt{k}} + B^{1/\sqrt{k}+6/(k+3)})
% \end{equation}
\begin{equation}
\label{eq:mainbigN}
\cR_0(N,B) \ll_{\aa,\tau,\epsi} B^{\frac{16}{3\sqrt{3k}} + \epsi} N^{\frac{24}{(3\tau)^{3/2}}-\frac{16}{(3k)^{3/2}}}\left(B^{2/\sqrt{k}} + B^{1/\sqrt{k}+6/(k+3)}\right)
\end{equation}
for any $\epsi > 0$.
\end{thm}

Note that, as in \cite{Heath-Brown09}, the determinant method discussed in Sections \ref{sec:parameterization} and \ref{sec:auxiliary} applies to any non-singular form. It is only in the later steps of the proof of Theorem \ref{thm:main} that we specialize to the case of a diagonal form.

The second result of the paper concerns the number $R_k(N)$ of representations of a positive integer $N$ as a sum of four $k$-th powers
\begin{equation}
\label{eq:4kthpowers}
x_1^k + x_2^k + x_3^k + x_4^k = N,
\end{equation}
where $x_i$ are \emph{non-negative} integers and $k \geq 3$. One easily proves, for example using Proposition \ref{prop:thue} below, that $R_k(N) = O_\epsi(N^{2/k + \epsi})$. Hooley \cite{Hooley78} has studied sums of four cubes, and proved the remarkable estimate $R_3(N) = O_\epsi(N^{11/18 + \epsi})$. Wisdom \cite{Wisdom98,Wisdom99} extended Hooley's methods to prove that $R_k(N) =O_\epsi(N^{11/(6k)+\epsi})$ for odd integers $k\geq 3$. Our result is the following:

\begin{thm}
\label{thm:wisdom}
%If $k \geq 3$, then the number $R_k(N)$ of solutions to \eqref{eq:4kthpowers} in non-negative integers $x_1,x_2,x_3,x_4$ satisfies
\[
R_k(N) \ll_{k,\epsi} N^{1/k + 2/k^{3/2} + \epsi}
\]
for any $\epsi > 0$.
\end{thm}

This estimate is non-trivial for $k\geq 5$, and sharper than Wisdom's for $k>5$. Theorem \ref{thm:wisdom} is proven in Section \ref{sec:wisdom}, as an easy corollary of the next result.
The estimate in Theorem \ref{thm:threepowers} was mentioned already in \cite{Heath-Brown09}, and is in principle contained in Salberger's work \cite{Salberger09}, but we shall give a proof in Section \ref{sec:wisdom} for the sake of completeness. 

\begin{thm}
\label{thm:threepowers}
Let $a_1,a_2,a_3,M$ be non-zero integers. Let $r_0(M,B)$ be the number of solutions $(x_1,x_2,x_3) \in \ZZ^3$ to the equation
\[
a_1 x_1^k + a_2 x_2^k + a_3 x_3^k = M
\]
satisfying $|x_i| \leq B$ and $a_i x_i^k \neq M$ for $i=1,2,3$. Then
\[
r_0(M,B) = O_{k,\epsi}(B^{2/\sqrt{k} + \epsi}).
\]
\end{thm}

In fact, with no extra work, the proof of Theorem \ref{thm:wisdom} yields the following more general result.

\begin{thm}
\label{thm:kl}
Let $k,\ell \geq 3$. Let $R_{k,\ell}(N)$ be the number of solutions to the equation
\begin{equation}
\label{eq:lth-and-kthpowers}
x_1^k + x_2^k + x_3^k + x_4^\ell = N
\end{equation}
in non-negative integers $x_i$. Then we have the estimate
\[
R_{k,\ell}(N) \ll_{k,\epsi} N^{1/\ell + 2/k^{3/2} + \epsi}
\]
for any $\epsi >0$.
\end{thm}

The corresponding trivial estimate is $N^{1/\ell + 1/k + \epsi}$. We also note that Wisdom \cite{Wisdom00} has proved that
\[
R_{3,4}(N)  = O_\epsi(N^{5/9+\epsi}) \text{ and } R_{3,5}(N) = O_\epsi(N^{47/90+\epsi}),
\]
bounds which are sharper than the ones given by Theorem \ref{thm:kl}.

\begin{notation*}
The following notation shall be used.
%\begin{notation*}
If $U \subseteq \AA^n$ is a locally closed subset, let $U(\ZZ)$ be the set of integral points in $U$. Then we define
\[
U(\ZZ,B) = U(\ZZ) \cap [-B,B]^n
\]
for any positive real number $B$, and
\[
\cN(U,B) = \#U(\ZZ,B).
\]
We shall also use the notation
\[
\cN_+(U,B) = \#(U(\ZZ)\cap [0,B]^n).
\]

%When we write $\phi(B) = O(\psi(B))$ or $\phi(B) \ll \psi(B)$ for two functions $\phi,\psi$, we mean that there is a constant $c_1$ such that $\phi(B) \leq c_1 \psi(B)$ for large enough $B$.
Finally, we adopt the following convention for the $O$- and $\ll$-notation. The implied constants are allowed to depend upon the coefficients of the polynomial $F$ under consideration (that is, on the $a_i$, in the case of a diagonal form) unless we indicate uniformity through the subscript $k$.
\end{notation*}

\begin{acknowledgement*}
This paper is a part of my doctoral thesis at Chalmers University of Technology. I thank my supervisor Per Salberger for suggesting the main topic of the paper, and for many helpful comments. I am also grateful to Tim Browning, who suggested that a result along the lines of Theorem \ref{thm:wisdom} could be obtained. Finally, I thank the organizers of the trimester program \emph{Diophantine Equations} at HIM in Bonn 2009, during which part of this work was done.
\end{acknowledgement*}

%%%%%%%%%%%%%%%%%%%%%%%%%%%%%%%%%%%%%%%%%%%%%%%%%%%%%%%%%%%%%%%%%%%%%%%%%%%%%%%%%%%%%%%%%%%%%%%%%%%%%%%%%%%%%%%

\section{Parameterization of points near projective surfaces}
\label{sec:parameterization}

In this section we generalize, in a completely straightforward fashion, some preparatory results in Heath-Brown's paper \cite{Heath-Brown09}. %Let $I$ denote the unit interval $[-1,1]$.
The proof of Lemma 1 in \cite{Heath-Brown09} generalizes readily to $\RR^4$, to yield our first lemma.

\begin{lemma}
\label{lem:goodcubes}
Let $F(x_1,x_2,x_3,x_4)$ be a non-singular homogeneous polynomial of degree $k$. There is a natural number $M_0$, depending only on $F$, with the following property: if the unit cube $[-1,1]^{3}$ is partitioned into smaller cubes
\[
[a_1,a_1+(M_0 M)^{-1}]\times [a_2,a_2+(M_0 M)^{-1}] \times [a_{3},a_{3}+(M_0 M)^{-1}],
\]
for some positive integer $M$, then the number of such cubes containing a solution $(t_1,t_2,t_{3})\in \RR^{3}$ to the inequality
\begin{equation}
\label{eq:F_ineq}
|F(t_1,t_2,t_{3},1)| \leq \frac{1}{M_0 M}
\end{equation}
is at most $O(M^{2})$. Moreover, if $S$ is such a cube containing a solution to \eqref{eq:F_ineq}, then for some index $i$ we have
\[
\left\vert \frac{\partial F}{\partial x_i}(a_1,a_2,a_3,1) \right\vert \gg 1.
\]
\end{lemma}

%%%%%%% Godtycklig dimension %%%%%%%%%%%%%%%%%%%%%%%%%%%%%%%%%%%%%%%%%%%%%%%%%%

% \begin{lemma}
% \label{lem:goodcubes}
% Let $F(x_1,\ldots,x_n)$ be a non-singular homogeneous polynomial of degree $k$. There is a natural number $M_0$, depending only on $F$, with the following property: if the unit cube $I^{n-1}$ is partitioned into smaller cubes
% \[
% [a_1,a_1+(M_0 M)^{-1}]\times \cdots \times [a_{n-1},a_{n-1}+(M_0 M)^{-1}],
% \]
% for some positive integer $M$, then the number of such cubes containing a solution $(t_1,\ldots,t_{n-1})\in \RR^{n-1}$ to the inequality
% \begin{equation}
% \label{eq:F_ineq}
% |F(t_1,\ldots,t_{n-1},1)| \leq \frac{1}{M_0 M}
% \end{equation}
% is at most $O(M^{n-2})$. Moreover, if $S$ is such a cube containing a solution to \eqref{eq:F_ineq}, then for some index $i$ we have
% \[
% \left\vert \frac{\partial F}{\partial x_i} \right\vert \gg 1,
% \]
% throughout $S$.
% \end{lemma}

%%%%%%%%%%%%%%%%%%%%%%%%%%%%%%%%%%%%%%%%%%%%%%%%%%%%%%%%%%%%%%%%%%%%%%%%%%%%%

Following Heath-Brown, let us call such a cube $S$ a ``good'' cube. Let us also call a solution $\ttt = (t_1,t_2,t_3)$ to \eqref{eq:F_ineq} a ``good'' point.
For a good cube, we can prove the following result. Again, the proof is an easy generalization of that of \cite[Lemma 2]{Heath-Brown09}.

\begin{lemma}
\label{lem:implicit}
Retaining the notation of the previous lemma, suppose that
\[
S=[a_1,a_1+(M_0 M)^{-1}]\times [a_2,a_2+(M_0 M)^{-1}] \times [a_3,a_3+(M_0 M)^{-1}]\]
is a good cube, and that $\left\vert \frac{\partial F}{\partial x_3}(a_1,a_2,a_3,1)\right\vert \gg 1$.

If $(t_1,t_2,t_3) = (a_1+u_1,a_2+u_2,a_3+u_3) \in S$, put
\[
w=F(t_1,t_2,t_3,1)-F(a_1,a_2,a_3,1).
\]
Then there exist, for each $m \in \NN$, two polynomials
\[
\Phi_m(u_1,u_2,w), \quad  \Psi_m(u_1,u_2,u_3,w),
\]
such that $\Phi_m$ has no constant term and $\Psi_m$ has no term of degree less than $m$, and such that the relation
\[
u_3 = \Phi_m(u_1,u_2,w) + u_3 \Psi_m(u_1,u_2,u_3,w)
\]
holds throughout $S$. Moreover, $\Phi_m$ and $\Psi_m$ have degree $O_{m}(1)$ and coefficients of size $O_m(1)$.
\end{lemma}

In other words, the lemma states that the relation
\[
F(a_1+u_1,a_2+u_2,a_3+u_3) = F(a_1,a_2,a_3)+w
\]
defines $u_3$, approximately, as a function of $u_1$, $u_2$ and $w$. It may thus be viewed as a form of the Implicit function theorem.

\section{Application of the determinant method}
\label{sec:auxiliary}

Let $F \in \ZZ[x_1,x_2,x_3,x_4]$ be a non-singular form of degree $k \geq 3$, $N$ a positive integer, and $B \geq 1$ a real number. Our aim in this section is to exhibit a set $\mathcal C$ of homogeneous polynomials $A_i \in \ZZ[x_1,x_2,x_3,x_4]$ of the same degree $\delta$, such that every solution $\xx \in \ZZ^4 \cap [-B,B]^4$ to the inequality
\begin{equation}
\label{eq:F_leq_N}
|F(x_1,x_2,x_3,x_4)| \leq N
\end{equation}
satisfies at least one of the equations $A_i(\xx)=0$. These polynomials shall be called auxiliary forms. Let $\cA(F,N,B,\delta)$ be the smallest possible cardinality of such a collection $\cC$ of auxiliary forms. This is a well-defined quantity since there are only finitely many solutions to \eqref{eq:F_leq_N}. Our arguments will conclude in two different estimates for $\cA(F,N,B,\delta)$. However, we begin with some considerations that apply to both situations, following closely the arguments in \cite{Heath-Brown09}.

%In the first case, we demand that $\delta < k$ and show that no more than $O(B^{\theta(k)})$ polynomials $A_i$ are needed, where $\theta(k) = O(k^{-1/2})$. The second result states that the number of polynomials is $O_{\epsi}(B^\epsi)$ for any $\epsi > 0$, provided $\delta \gg_{\epsi} 1$.
Since we are only interested in the order of growth of $\cA(F,N,B,\delta)$ as a function of $B$, it is clearly enough to consider solutions to \eqref{eq:F_leq_N} for which
\[
|x_4| \geq \max(|x_1|,|x_2|,|x_3|).
\]
We may even restrict ourselves to solutions satisfying
\begin{equation}
\label{eq:hollowbox}
B/2 < \max_i(|x_i|) = |x_4| \leq B,
\end{equation}
deducing the final estimate from this case by dyadic summation.

If we assume that
\begin{equation}
\label{eq:M_leq}
M \leq (B/2)^k M_0^{-1} N^{-1},
\end{equation}
then every solution $\xx$ to \eqref{eq:F_leq_N} and \eqref{eq:hollowbox}
produces a good point
\[
\ttt = (x_1/x_4,x_2/x_4,x_3/x_4),
\]
which by Lemma \ref{lem:goodcubes} lies in some good cube. Thus, let
\[
S=[a_1,a_1+(M_0 M)^{-1}]\times [a_2,a_2+(M_0 M)^{-1}] \times [a_3,a_3+(M_0 M)^{-1}]
\]
be a good cube, and let $R=\{\xx^{(j)},1\leq j\leq J\}$ be the set of $\xx\in\ZZ^4$ satisfying \eqref{eq:F_leq_N} and \eqref{eq:hollowbox}, and such that $\ttt \in S$.
For now, let $\delta$ be any positive integer, and let $s=\binom{\delta+3}{3}$ be the number of different monomials in $x_1,x_2,x_3,x_4$ of degree $\delta$.
Consider the $s \times J$-matrix
\[
\cA = \big( f_i(\xx^{(j)}) \big),
\]
where $f_i$ runs over all monomials of degree $\delta$. We shall prove that it is possible to choose $M$ in such a way that $\rank \cA < s$. This implies the existence of a homogeneous polynomial $A(x_1,x_2,x_3,x_4)$ of degree $\delta$ vanishing at all the $\xx^{(j)}$.

If $J<s$, we are done. Otherwise, we proceed by choosing a subset of $R$ of cardinality $s$, without loss of generality $\{\xx^{(1)},\dotsc,\xx^{(s)}\}$, and evaluating the corresponding $s\times s$-subdeterminant
\[
\Delta_1 = \det \big( f_i(\xx^{(j)}) \big)_{1\leq i,j \leq s}.
\]
Our aim is to prove that $|\Delta_1| < 1$. In that case, being an integer, $\Delta_1$ has to vanish.
We have
\begin{equation}
\label{eq:Delta1}
\Delta_1 = \prod_{j=1}^s (x^{(j)}_4)^\delta \Delta_2 \ll B^{s\delta}|\Delta_2|,
\end{equation}
where $\Delta_2 =  \det \big( f_i(t^{(j)}_1,t^{(j)}_2,t^{(j)}_3,1) \big)$.

At this point, we make the "variable change" suggested by Lemma \ref{lem:implicit}. Suppose, without loss of generality, that $\left\vert \partial F/\partial x_3(a_1,a_2,a_3,1) \right\vert \gg 1$. For $(t_1,t_2,t_3) \in S$, let $u_1,u_2,u_3,w$ be as in Lemma \ref{lem:implicit}. Furthermore, let $\xi = F(t_1,t_2,t_3,1) = w+ F(a_1,a_2,a_3,1)$. For any monomial $f_i$, we have 
\[
f_i(t_1,t_2,t_3,1) = G_i(u_1,u_2,u_3),
\]
where $G_i=G_{i,S}$ is a polynomial of degree $\delta$. For any $m \in \NN$, an application of Lemma \ref{lem:implicit} now yields
\begin{align*}
f_i(t_1,t_2,t_3,1) &= G_i(u_1,u_2,\Phi_m(u_1,u_2,w) + u_3 \Psi_m(u_1,u_2,u_3,w))\\
&= Q_i(u_1,u_2,w) + u_3 H_i(u_1,u_2,u_3,w),
\end{align*}
where $Q_i = Q_{i,S}$ and $H_i = H_{i,S}$ are polynomials of degree $O_m(1)$ and all terms of $H_i$ have degree at least $m$.

Now, if $\ttt$ is a good point, we have
$u_i \ll M^{-1}$ and $\xi \ll M^{-1}$, and thus also $w \ll M^{-1}$, so we get
\begin{align*}
f_i(t_1,t_2,t_3,1) &= g_i(u_1,u_2,\xi) + O_m(M^{-(m+1)}),
\end{align*}
for some polynomials $g_i$ of degree $O_m(1)$. Here, $g_i$ depends on the chosen cube $S$, but the size of the coefficients of $g_i$ is bounded in terms of $m$. We conclude that
\begin{equation}
\label{eq:Delta2}
\Delta_2 = \Delta_3 + O_m(M^{-(m+1)}),
\end{equation}
where $\Delta_3 =  \det \big( g_i(u^{(j)}_1,u^{(j)}_2,\xi^{(j)}) \big)$.

To estimate the determinant $\Delta_3$, we shall use a variant of the argument in \cite{Bombieri-Pila} where we take into account the fact that one of the variables, $\xi$, takes only small values. Let us recall the notation from \cite{Heath-Brown09}: let $n,D,H$ be positive integers. Given real numbers $0 \leq X_1,\dotsc,X_n \leq 1$, we define the size $\Vert m_i \Vert$ of a monomial $m_i(x_1,\dotsc,x_n) = x_1^{\alpha_1}\dotsb x_n^{\alpha_n}$ by
\[
\Vert m_i \Vert = X_1^{\alpha_1}\dotsb X_n^{\alpha_n}.
\]
Furthermore, we enumerate the monomials $m_1,m_2,\dotsc$ in $x_1,\dotsc,x_n$ in such a way that $\Vert m_1 \Vert \geq \Vert m_2 \Vert \geq \dotsb$. Finally, by abuse of notation, by the height $\Vert f \Vert$ of a polynomial $f \in \CC[x_1,\dotsc,x_n]$ we mean the maximum modulus of its coefficients. Heath-Brown proves the following result.

\begin{lemma}[{\cite[Lemma 3]{Heath-Brown09}}]
\label{lem:vandermonde}
Let $f_1,\ldots,f_H \in \CC[x_1,\ldots,x_n]$ be polynomials of degree at most $D$. Let $\xx^{(1)},\dotsc,\xx^{(H)} \in \CC^n$ satisfy $|x^{(j)}_i| \leq X_i$ for all $i$ and $j$. Then we have the estimate
\[
\det(f_i(\xx^{(j)}))_{1\leq i,j \leq H} \ll_{H,D} (\max_i \Vert f_j \Vert)^H \prod_{j=1}^H \Vert m_i \Vert.
\]
\end{lemma}

In the application of Lemma \ref{lem:vandermonde}, we take $X_1 = X_2 = (M_0 M)^{-1}$ and $X_3 = N (B/2)^{-k}$, according to our a priori bounds for $|u_1|$, $|u_2|$ and $|\xi|$, respectively.

Let the monomials $m_i(u_1,u_2,\xi)$ be defined as above. The strategy of our method is to ensure that for small $i$, $m_i$ does not contain a positive power of $\xi$. In that way our determinant will behave almost as if we were considering points on a projective surface instead of points on an affine threefold. Our approach now differs from that in \cite{Heath-Brown09} in that we allow $\delta$ to grow as large as required to obtain an optimal bound, whereas Heath-Brown only considers $\delta < k$. Thus, suppose that
\begin{equation}
\label{eq:alpha}
(M_0 M)^\alpha = N^{-1} (B/2)^k,
\end{equation}
where $\alpha$ is to be chosen properly. Then Lemma \ref{lem:vandermonde} yields an estimate
\[
\Delta_3 \ll_{\delta,m} \prod_{i=1}^s \Vert m_i \Vert \ll_{\delta,m} M^{-f} ,
\]
where $f = \sum (n_1+n_2+\alpha n_3)$, the sum being taken over all non-negative integers $n_1,n_2,n_3$ such that $u_1^{n_1} u_2^{n_2} \xi^{n_3}$ occurs among the monomials $m_i$, $1\leq i \leq s$.

To estimate $f$, we need to determine the first monomials $m_1,\dotsc,m_s$. Thus, suppose that $m_s(u_1,u_2,\xi) = u_1^{e_1} u_2^{e_2} \xi^{e_3}$, and let
\[
\nu = e_1 + e_2 + \alpha e_3,
\]
so that $\Vert m_s \Vert = (M_0 M)^\nu$.
To determine the relationship between $\delta$ and $\nu$, we note that
\begin{equation}
\label{eq:s_squeeze}
\sum_{\substack{n_1,n_2,n_3 \geq 0\\
n_1+n_2+\alpha n_3 \leq \nu -1}} 1
< s \leq
\sum_{\substack{n_1,n_2,n_3 \geq 0\\
n_1+n_2+\alpha n_3 \leq \nu}} 1.
\end{equation}
The left sum in \eqref{eq:s_squeeze}, i.e. the number of integer points inside the tetrahedron $T_1 \subset \RR^3$ defined by the inequalities
\[
x\geq 0,\ y\geq 0,\ z \geq 0,\ x + y + \alpha z \leq \nu -1,
\]
can be interpreted as the volume of the three-dimensional body
\begin{gather*}
S_1 = \bigcup_{\substack{n_1,n_2,n_3 \geq 0\\
n_1+n_2+\alpha n_3 \leq \nu -1}}
[n_1,n_1+1]\times[n_2,n_2+1]\times[n_3,n_3+1].
%S_2 = \bigcup_{\substack{n_1,n_2,n_3 \geq 0\\
%n_1+n_2+\alpha n_3 \leq \nu}}
%[n_1,n_1+1]\times[n_2,n_2+1]\times[n_3,n_3+1].
\end{gather*}
Since $T_1 \subset S_1 \subset T_2$, where $T_2$ is the tetrahedron
\[
x\geq 0,\ y\geq 0,\ z \geq 0,\ x-1 + y-1 + \alpha (z-1) \leq \nu -1,
\]
we get the estimate
\begin{equation}
\label{eq:vol(S_1)}
\frac{(\nu - 1)^3}{6 \alpha} = \vol(T_1) < \#(T_1 \cap \ZZ^3) < \vol(T_2) = \frac{(\nu+1+\alpha)^3}{6 \alpha}.
\end{equation}

Similarly, the right sum in \eqref{eq:s_squeeze} is the number of integer points inside the tetrahedron $T_3$ defined by the inequalities
\[
x\geq 0,\ y\geq 0,\ z \geq 0,\ x-1 + y-1 + \alpha (z-1) \leq \nu,
\]
so that,
\[
\sum_{\substack{n_1,n_2,n_3 \geq 0\\
n_1+n_2+\alpha n_3 \leq \nu}} 1
\leq \vol(T_3) = \frac{(\nu+2+\alpha)^3}{6\alpha}.
\]
By the definition of $s$ we conclude that
\begin{equation}
\label{eq:delta}
\delta = \frac{\nu}{\alpha^{1/3}} + O_{\alpha}(\nu^{2/3}).
\end{equation}

A lower bound for $f$ is given by the sum
\[
\tilde{f} = \sum_{(n_1,n_2,n_3) \in T_1\cap\ZZ^3} (n_1 + n_2 + \alpha n_3).
\]
We can estimate $\tilde{f}$ from below by considering the integral
\[
I = \int_{T_1} (x+y+\alpha z)\, dx\,dy\,dz = \frac{(\nu-1)^4}{8\alpha}.
\]
Since $T_1 \subset S_1$, we have
\begin{align*}
I &< \sum_{(n_1,n_2,n_3) \in T_1\cap\ZZ^3} (n_1+1+n_2+1+\alpha(n_3+1))\\
  &= \tilde f + (2+\alpha) \#(T_1 \cap \ZZ^3).
\end{align*}
By \eqref{eq:vol(S_1)} we conclude that
\begin{equation}
\label{eq:f}
f > \frac{(\nu-1)^4}{8\alpha} - \frac{(2+\alpha)(\nu+1+\alpha)^3}{6\alpha} = \frac{\nu^4}{8\alpha} + O_\alpha(\nu^3).
\end{equation}

Tracing our steps back to the estimates \eqref{eq:Delta1} and \eqref{eq:Delta2}, and choosing $m = f = O_{a,\delta}(1)$, we get the estimate
\[
\Delta_1 \ll_{F,\delta,\alpha} B^{s \delta} (M_0 M)^{-f} \ll_{F,\delta,\alpha} B^\beta,
\]
where, upon recalling the relation \eqref{eq:alpha}, we have
\begin{align*}
\beta &= s \delta - \frac{f}{\alpha}\left(k-\frac{\log N}{\log B}\right).
\end{align*}
Using \eqref{eq:s_squeeze}, \eqref{eq:delta} and \eqref{eq:f} this implies that
\begin{equation}
\label{eq:beta}
%\begin{split}
\beta = \left( \frac{1}{6} - \left(k-\frac{\log N}{\log B}\right) \frac{1}{8\alpha^{2/3}} \right) \delta^4 + O_\alpha(\delta^3)
%\end{split}.
\end{equation}

Let us first consider the case where $N$ remains fixed as $B \to \infty$. Given $\epsi > 0$, choose $\lambda = \lambda(k,\epsi) > 0$ small enough that
\[
\frac{16}{3\sqrt{3k}} (1- \lambda)^{-1} \leq \frac{16}{3\sqrt{3k}} + \epsi
\quad \text{and}\quad
(1-\lambda)\left(\frac{3 k}{4} \right)^{3/2} > 1,
\]
and put
\[
\alpha = (1 - \lambda) \left(\frac{3 k}{4} \right)^{3/2},
\]
so that
\[
\frac{2k}{\alpha} = \frac{16}{3\sqrt{3k}} (1- \lambda)^{-1} \leq \frac{16}{3\sqrt{3k}} + \epsi.
\]

Then there is a positive constant $c_1=c_1(k,\epsi)$ such that
\[
\frac{1}{6} - \left(k-\frac{\log N}{\log B}\right) \frac{1}{8\alpha^{2/3}} < -c_1,
\]
that is, the leading coefficient in \eqref{eq:beta} is negative, as soon as $B \geq B_0 = B_0(N,k,\epsi)$. Thus, we have $\beta < -1$, say, as soon as $\delta \geq \delta_0 = \delta_0(k,\epsi)$.

Assume now that $B \geq B_0$ and, in addition, that 
\[
M = M_0^{-1} N^{-1/\alpha} (B/2)^{k/\alpha}
\]
is an integer (we may clearly restrict ourselves to such values of $B$). As $\alpha > 1$, the requirement \eqref{eq:M_leq} is fulfilled, so the above arguments are indeed valid. In particular, for $\delta \geq \delta_0$ we get
\[
\Delta_1 \ll_{F,\epsi} B^{-1},
\]
so that $|\Delta_1| < 1$ as soon as $B \gg_{F,\epsi} 1$.
In this situation, as already explained, we obtain an auxiliary form of degree $\delta$ for each good cube $S$. By Lemma \ref{lem:goodcubes}, the total number of good cubes is
\[
O_F(M^2) = O_{F,N}(B^{2k/\alpha}) = O_{F,N}(B^{16/{3\sqrt{3k}}+\epsi}),
\]
and so this constitutes an upper estimate for $\cA(F,N,B,\delta)$. Thus we can summarize our findings so far in the following result:

\begin{prop}
\label{prop:auxiliaryform}
Let $F \in \ZZ[x_1,x_2,x_3,x_4]$ be a non-singular homogeneous polynomial of degree $k\geq 3$. Let $N \in \NN$ be  given. Then for any $\epsi > 0$, there is an integer $\delta$, depending only on $k$ and $\epsi$, such that
\[
\cA(F,N,B,\delta) = O_{N,\epsi}(B^{16/(3\sqrt{3k}) + \epsi}).
\]
\end{prop}

Next, let us allow $N$ to vary as $B$ grows. In this situation, we shall prove the following estimate.
\begin{prop}
\label{prop:auxiliaryformbigN}
Let $F \in \ZZ[x_1,x_2,x_3,x_4]$ be a non-singular homogeneous polynomial of degree $k\geq 3$. Let $N \in \NN$ be  given such that $N=O(B^{k-\tau})$, where $4/3 < \tau < k$. Then for any $\epsi > 0$, there is an integer $\delta$, depending only on $k$,$\tau$ and $\epsi$, such that
\[
\cA(F,N,B,\delta) = O_{\tau,\epsi}\left(B^{\frac{16}{3\sqrt{3k}} + \epsi} N^{\frac{24}{(3\tau)^{3/2}}-\frac{16}{(3k)^{3/2}}}\right).
\]
% \[
% \cA(F,N,B,\delta) = O_{\mu,\epsi}\left(B^{\frac{16}{3\sqrt{3k}} + \epsi} N^{\frac{8}{\sqrt{3}k^{3/2}}\left(\frac{\sqrt{2}}{(1-\mu)}-\frac{1}{3}\right)}\right)
% \]
\end{prop}

To prove Proposition \ref{prop:auxiliaryformbigN}, let $t$ be an arbitrary real number in the interval $\tau \leq t < k$, and put
\begin{equation}
\label{eq:alphadef}
\alpha = (1 - \lambda) \left(\frac{3t}{4}\right)^{3/2},
\end{equation}
where $\lambda > 0$ is to be chosen properly in terms of $k$, $\epsi$ and $\tau$. Put $\gamma = k - t$. Suppose now that $\frac{1}{2}B^\gamma < N \leq B^\gamma$, say, and that $M = M_0^{-1} N^{-1/\alpha} (B/2)^{k/\alpha}$ is an integer. It clearly suffices to prove our estimate for such $N$ and $B$, as $t$ runs over the interval $[\tau,k)$, allowing for an implicit constant $O_{k,\epsi,\tau}(1)$ in our final estimate.
Now we have
\[
\frac{1}{6} - \left(k-\frac{\log N}{\log B}\right) \frac{1}{8\alpha^{2/3}} \leq \frac{1}{6} \left( 1 - (1-\lambda)^{-2/3} \right) \leq -c_2,
\]
where $c_2 = c_2(\lambda) > 0$, so that
\[
\beta \leq -c_2 \delta^4 + O_{k}(\delta^3).
\]
As above, this implies that $|\Delta_1| < 1$ as soon as $\delta \gg_{k,\lambda} 1$ and $B \gg_{F,\lambda} 1$, in which case
\begin{equation}
\label{eq:A}
\cA(F,N,B,\delta) \ll_F M^2.
\end{equation}

We shall now see how to choose $\lambda$ to obtain the optimal estimate. From \eqref{eq:alphadef} we have
\[
\frac{1}{\alpha} =  \left(\frac{4}{3t}\right)^{3/2} (1-\lambda)^{-1} = \left(\frac{4}{3k}\right)^{3/2} (1-\lambda)^{-1}\left(1-\frac{\gamma}{k}\right)^{-3/2}.
\]
By Taylor expansion of the function $x \mapsto (1+x/t)^{3/2}$ at $x=0$, one sees that
\[
\left(1-\frac{\gamma}{k}\right)^{-3/2} = \left(1+\frac{\gamma}{t}\right)^{3/2} \leq 1 + \frac{3k^{1/2}}{2 t^{3/2}} \gamma.
\]
Furthermore, we may assume that $\lambda < 1/2$, so that $(1 -\lambda)^{-1} < 1 + 2\lambda$. Thus, for any $\epsi > 0$, we have
\[
\frac{8}{3^{3/2}k^{3/2}} < \frac{1}{\alpha} < \frac{8}{3^{3/2}k^{3/2}} + \frac{\epsi}{2k} + \gamma \frac{4}{3^{1/2}\tau^{3/2}k},
\]
upon choosing $\lambda \ll_{k,\epsi,\tau} 1$. (Note that $\lambda$ does not depend on $\gamma$.) Using this estimate and the fact that $\gamma \leq \log (2N)/\log B$, we get
\[
B^{2k/\alpha} \ll B^{\frac{16}{3\sqrt{3k}} + \epsi} N^{\frac{24}{(3\tau)^{3/2}}},
\]
and hence
\[
M^2 < B^{2k/\alpha} N^{-2/\alpha} 
\ll B^{\frac{16}{3\sqrt{3k}} + \epsi} N^{\frac{24}{(3\tau)^{3/2}}-\frac{16}{(3k)^{3/2}}}.
\]
In view of \eqref{eq:A}, this establishes the estimate in Proposition \ref{prop:auxiliaryformbigN}.

%%%%%%%%%%%%%%%%%%%%%%%%%%%%%%%%%%%%%%%%%%%%%%%%%%%%%%%%%%%%%%%%%%%%%%%%%%%%%%%%%%%%%%%%%%%%%%%%%%%%%%%

\section{Curves of low degree on Fermat threefolds}
\label{sec:fermat}

The following result was proven by Salberger:

\begin{thm}[{\cite[Thm. 8.1]{Salberger09}}]
\label{thm:fermat}
Let $K$ be an algebraically closed field of characteristic $0$, and let $(a_0,\dotsc,a_n)$ be an $(n+1)$-tuple of non-zero elements of $K$. Let $X \subset \PP^n_K$ be the Fermat hypersurface given by $a_0 x_0^k + \dotsb + a_n x_n^k=0$. Suppose that $C \subset X$ is an integral curve of degree $e$ that does not lie on any other hypersurface defined by a diagonal form $b_0 x_0^k + \dotsb b_n x_n^k$. Then the following holds:
\[
(n+1)(k-(n-1)) \leq n d + \frac{n(n-1)(e-3)}{2}.
\]
\end{thm}

From this, Salberger draws the following conclusion:

\begin{thm}[{\cite[Thm. 8.4]{Salberger09}}]
\label{thm:salberger8.4}
Let $K$ be an algebraically closed field of characteristic $0$. Let $(a_0,a_1,a_2,a_3)$ be a quadruple of non-zero elements of $K$ and $X \subset \PP^3_K$ the Fermat hypersurface given by $a_0 x_0^k + a_1 x_1^k + a_2 x_2^k + a_3 x_3^k=0$, where $k \geq 3$. Let $C \subset X$ be an integral curve. If
\[
\deg C < (k+1)/3,
\]
then $C$ is one of the $3 k^2$ lines defined on $X$ by the equation $a_0 x_0^k + a_j x_j^k = 0$, where $j=1,2$ or $3$. 
\end{thm}

We shall now derive an analogous statement for Fermat threefolds in $\PP^4$. Let $K$ be an algebraically closed field of characteristic $0$. Let $(a_0,\dotsc,a_4)$ be a quintuple of non-zero elements of $K$ and $X \subset \PP^4_K$ the Fermat hypersurface given by $a_0 x_0^k + \dotsb + a_4 x_4^k=0$, where $k \geq 4$. For any partition
\[
\{0,1,2,3,4\} = \{i_0,i_1\} \cup \{i_2,i_3,i_4\},
\]
the subvariety of $X$ defined by 
\[
a_{i_0}x_{i_0}^k + a_{i_1} x_{i_1}^k = a_{i_2}x_{i_2}^k + a_{i_3} x_{i_3}^k +a_{i_4} x_{i_4}^k = 0
\]
is covered by a one-dimensional family of lines, called \emph{standard lines} \cite[\S 2]{Debarre}. It is well-known \cite[Ex. 2.5.3]{Debarre} that all lines contained in $X$ are standard if $k\geq 4$. The following result strengthens that statement to say that all curves on $X$ of sufficiently low degree are in fact standard lines.

\begin{prop}
\label{prop:fermat}
% Let $K$ be an algebraically closed field of characteristic $0$. Let $(a_0,\dotsc,a_4)$ be a quintuple of non-zero elements of $K$ and $X \subset \PP^4_K$ the Fermat hypersurface given by $a_0 x_0^k + \dotsb + a_4 x_4^k=0$, where $k \geq 4$. 
Let $C \subset X$ be an integral curve. If
\[
\deg C < (k+3)/6,
\]
then $C$ is a standard line. 
% and there is a partition
% \[
% \{0,1,2,3,4\} = \{i_0,i_1\} \cup \{i_2,i_3,i_4\}
% \]
% such that
% \[
% a_{i_0}x_{i_0}^k + a_{i_1} x_{i_1}^k = a_{i_2}x_{i_2}^k + a_{i_3} x_{i_3}^k +a_{i_4} x_{i_4}^k = 0
% \]
% on $C$.
\end{prop}

% \begin{rem*}
% These lines are called standard lines. The proposition extends (as soon as $k \geq 10$) the well-known fact that all lines contained in $X$ are standard \cite[Ex. 2.5.3]{Debarre}.
% \end{rem*}

\begin{proof}
We shall first prove the following statement
\begin{itemize}
\item[(I)]
There exists a three-element subset $\{i_0,i_1,i_2\}$ of $\{0,1,2,3,4\}$ and a diagonal form $c_{0} x_{i_0}^k + c_{1} x_{i_1}^k + c_{2} x_{i_2}^k$ that vanishes on $C$, with all $c_i \neq 0$.
\end{itemize}
By Theorem \ref{thm:fermat} there is a diagonal form $b_0 x_0^k + \dotsb + b_4 x_4^k$, linearly independent from $a_0 x_0^k + \dotsb + a_4 x_4^k$, that vanishes on $C$. Choosing a suitable linear combination of the two forms, we can assume that either one or two of the coefficients $b_i$ vanish. If there are exactly three non-zero coefficients we are done, so let us assume that there are four. By permuting the variables, we assume for the sake of simplicity that $b_4=0$ and $b_i \neq 0$ for $i < 4$.

Next, let $\pi: \PP^4 \dashrightarrow \PP^3$ be the rational map given by projection onto the first four coordinates. Let $Y \subset \PP^3$ be the hypersurface given by $b_0 x_0^k + \dotsb + b_3 x_3^k=0$. Then the image $\pi(C)$ is an irreducible curve $C' \subset Y$. Indeed, the image is either a point or a curve, but the first alternative would imply that $C$ were a line containing the point $(0:0:0:0:1)$, which would contradict the fact that $a_4 \neq 0$. Furthermore we have $\deg C' \leq (k+3)/6 \leq (k+1)/3$, so by Theorem \ref{thm:salberger8.4}, $C'$ is a standard line. In other words, there is a partition $\{0,1,2,3\} = \{j_0,j_1\} \cup \{j_2,j_3\}$ such that $b_{j_0} x_{j_0}^k + b_{j_1} x_{j_1}^k = b_{j_2} x_{j_2}^k + b_{j_3} x_{j_3}^k = 0$ on $C$. Choosing a suitable linear combination of the forms $a_0 x_0^k + \cdots + a_4 x_4^k$, $b_{j_0} x_{j_0}^k + b_{j_1} x_{j_1}^k$ and $b_{j_2} x_{j_2}^k + b_{j_3} x_{j_3}^k$, we get (I).

% so by the case $n=3$ of Theorem \ref{thm:fermat}, there is a form $d_0 x_0^k + \dotsb d_3 x_3^k$ that vanishes on $C'$.
%
% There are now two possibilities. First, suppose that $b_i d_j - b_j d_i = 0$ for all $0\leq i<j \leq 3$. Then there is a form $c_0 x_0^k + \dotsb + c_3 x_3^k$ with exactly one $c_i = 0$, in the linear pencil spanned by $b_0 x_0^k + \dotsb b_3 x_3^k$ and $d_0 x_0^k + \dotsb d_3 x_3^k$. This form vanishes on $C'$ and thus also on $C$, so the assertion (I) follows.
%
% On the other hand, suppose that $b_i d_j - b_j d_i = 0$ for some $0\leq i<j \leq 3$, say $b_0 d_2 - b_2 d_0 = 0$ without loss of generality. Then there are two forms $b_0' x_0^k + b_2' x_2^k$, $d_1' x_1^k + d_3' x_3^k$ that vanish on $C'$, with all coefficients non-zero. In other words $C'$ is a line, given by $x_0 = \alpha x_2$, $x_1 = \beta x_3$ for some $\alpha, \beta \neq 0$. But then these two relations hold on $C$ as well, and inserting them into the equation for $X$ we see that
% \[
% a_2' x_2^k + a_3' x_3^k + a_4 x_4^k = 0
% \]
% on $C$, where $a_2' = a_2 + a_0 \alpha^k$ and $a_3' = a_3 + a_1 \beta^k$. This immediately yields (I) if $a_2',a_3'\neq 0$, but it is easy to see why (I) follows even if one or both of these coefficients vanish.

Having proven (I), we may assume, by permuting the variables, that the form $c_0 x_0^k + c_1 x_1^k+c_2 x_2^k$ vanishes on $C$. We shall prove that $C$ is a line. To this end, let $\pi_1:\PP^4 \dashrightarrow \PP^2$ be the projection onto the first three coordinates. Let $Z \subset \PP^2$ be the subvariety given by $c_0 x_0^k + c_1 x_1^k+c_2 x_2^k=0$. As above, $\pi_1(C)$ is either a point or an irreducible curve contained in $Z$. But this curve would have degree less than $(k+3)/6$, which would contradict the fact that $Z$ is an irreducible curve of degree $k$. Therefore $\pi_1(C)$ is a single point, say $(y_0:y_1:1)$ without loss of generality.

This means that $C$ is contained in the plane $\Pi_1 \subset \PP^4$ given by the equations $x_0 - y_0 x_2 = x_1 - y_1 x_2 = 0$. Inserting this into the equation for $X$, we infer that
\[
a_2' x_2^k + a_3 x_3^k + a_4 x_4^k = 0
\]
on $C$, where $a_2' = a_0 y_0^k + a_1 y_1^k + a_2$. If $a_2' = 0$, we infer that $C$ is one of the $k$ lines given by the equations
\[
a_3 x_3^k + a_4 x_4^k = x_0 - y_0 x_2 = x_1 - y_1 x_2 = 0.
\]
If $a_2' \neq 0$, then by the same argument as above, $C$ is mapped to a point by the projection $\pi_2: \PP^4 \dashrightarrow \PP^2$ onto the last three coordinates, which implies that $C$ is contained in some plane $\Pi_2 \subset \PP^4$, necessarily distinct from $\Pi_1$. $C$ is then the line $\Pi_1 \cap \Pi_2$. It would now be easy to proceed by showing that $C$ is one of the standard lines, but as remarked above, this is a known result.
\end{proof}

%%%%%%%%%%%%%%%%%%%%%%%%%%%%%%%%%%%%%%%%%%%%%%%%%%%%%%%%%%%%%%%%%%%%%%%%%%%%%%%%%%%%%%%%%%%%%%%%%%%%%%%%%

\section{Counting integral points on affine surfaces}
\label{sec:integralpoints}

From now on we consider the case of a diagonal form. Thus, let
\[
F(x_1,x_2,x_3,x_4) = a_1 x_1^k + a_2 x_2^k + a_3 x_3^k + a_4 x_4^k,
\]
where $a_i$ are non-zero integers, let $N$ be a positive integer and $B\geq 1$ a real number. Furthermore, let $X \subset \AA^4$ be the hypersurface defined by
\[
F(x_1,x_2,x_3,x_4) = N.
\]
%The quantity we wish to estimate is then $\cR(N,B) = \cN(X,B)$.
Let $V_i \subset \AA^4$, for $1 \leq i \leq 4$, be the closed subvariety defined by
\[
a_i x_i^k = N, \quad \sum_{j \neq i} a_j x_j^k = 0
\]
and $W_{i,j}$, for $1 \leq i < j \leq 4$, be defined by
\[
a_i x_i^k + a_j x_j^k = N, \quad \sum_{\ell \neq i,j} a_\ell x_\ell^k = 0.
\]
As is shown in Section \ref{sec:fermat}, the algebraic set
\[
V = \left( \bigcup_{1\leq i\leq 4} V_i \right) \cup \left( \bigcup_{1\leq i < j \leq 4} W_{i,j} \right)
\]
is precisely the union of all lines on $X$. The quantity we wish to estimate is then $\cR_0(N,B) = \cN(X_0,B)$, where $X_0 := X \setminus V$.

By Proposition \ref{prop:auxiliaryform} we know that every $\xx \in X(\ZZ,B)$ satisfies
\begin{equation}
\label{eq:subvariety}
F(x_1,x_2,x_3,x_4) = N,\quad A_i(x_1,x_2,x_3,x_4)=0,
\end{equation}
for one of $O_{N,\epsi}(B^{16/(3\sqrt{3k})+\epsi})$ forms $A_i$ of degree $O_{\epsi}(1)$.

%, if $k \geq 41$. If $k \leq 40$, then $9/\sqrt{2k} > 1$, so instead we may trivially take $A_i$ to range over the collection of $O(B)$ linear (although non-homogeneous) polynomials given by $x_0 - a$ for all integers $-B \leq a \leq B$.

\begin{rem*}
%Although Theorem \ref{thm:main}, strictly speaking, is a corollary to Theorem \ref{thm:mainbigN}, we 
We shall only write out the proof of Theorem \ref{thm:main}. If we would supply the more precise estimate of Proposition \ref{prop:auxiliaryformbigN} at this point, we would obviously get a proof of Theorem \ref{thm:mainbigN}.
\end{rem*}

Let $\tilde Y \subset \AA^4$ be any one of the varieties defined by \eqref{eq:subvariety}. Since $A_i$ is homogeneous, it cannot vanish entirely on $X$, so the dimension of $\tilde Y$ is $2$. Let $Y$ be an irreducible component of $\tilde Y$. As we shall see shortly, we may assume that $Y$ is in fact geometrically irreducible. Then, as $\bar Y$ is a closed subvariety of the non-singular hypersurface $\bar X$, where $\bar X, \bar Y \subset \PP^4$ denote the respective projective closures, it follows from the Noether-Lefschetz theorem \cite[pp. 180-1]{Hartshorne70} that the degree $d$ of $\bar Y$ is divisible by $k$.

It is then possible \cite[Prop. 6.2]{Marmon10} to find an affine projection $\pi : Y \to \AA^3$ that is birational onto its image, and such that integral points of height at most $B$ are mapped onto integral points of height at most $cB$ for some constant $c \ll_k 1$. Then $W = \pi(Y) \subset \AA^3$ is an irreducible closed subvariety of dimension $2$ and degree $d$, and $\pi^{-1}(\xx)$ consists of at most $d$ points for any $\xx \in W$.

Now we use the new version of the determinant method developed by Salberger. For the sake of convenience, we recall the following result from \cite{Salberger09}.

\begin{thm}[{\cite[Thm. 7.2]{Salberger09}}]
\label{thm:Salberger7.2}
Let $X \subset \AA^3_\QQ$ be a geometrically integral surface of degree $d$. Then there is a collection of $O_{d,\epsi}(B^{1/\sqrt{d}+\epsi})$ geometrically integral curves $D_\lambda \subset X$, $\lambda \in \Lambda$, of degree $O_d(1)$, such that
\[
\cN \Big(X \setminus \bigcup_{\lambda \in \Lambda} D_\lambda, B \Big) = O_{d,\epsi}(B^{2/\sqrt{d}+\epsi}).
\]
\end{thm}

From Theorem \ref{thm:Salberger7.2} we infer that there is a collection of $O_{k,\epsi}(B^{1/\sqrt{k}+\epsi})$ irreducible curves on $W$ of degree $O_{k}(1)$ such that all but $O_{k,\epsi}(B^{2/\sqrt{k}+\epsi})$ points of $W(\ZZ,B)$ lie on one of these curves. Pulling these curves and points back by $\pi$, we get $O_{k,\epsi}(B^{1/\sqrt{k}+\epsi})$ irreducible curves of bounded degree on $Y$, the union of which contains all but $O_{k,\epsi}(B^{2/\sqrt{k}+\epsi})$ points of $Y(\ZZ,B)$.

Concerning the case where $Y$ is integral but not geometrically integral, we can say more. Indeed, one can argue as in \cite[Proof of Thm. 2.1]{Salberger05} to conclude that all rational points on $Y$ lie on a single curve, the sum of the degrees of the irreducible components of which is bounded in terms of $k$. Thus these irreducible components can be absorbed in the collection of curves and points of the previous paragraph.

To investigate the nature of such a curve, we shall use Proposition \ref{prop:fermat} on the hypersurface
\[
\bar X = \{ -N x_0^k + a_1 x_1^k + a_2 x_2^k + a_3 x_3^k + a_4 x_4^k = 0 \} \subset \PP^4.
\]
Any irreducible curve on $X$ of degree less than $(k+3)/6$ gives rise to an irreducible curve of the same degree on $\bar X$, and must therefore in fact be one of the lines in $V$.

%Let $\psi(k) = \min\{1,9/\sqrt{2k}\}$.
Since the number of irreducible components of a surface $\tilde Y$ as above is bounded in terms of $k$, we conclude that
\begin{equation}
\label{eq:auxiliary}
X_0(\ZZ,B) \subseteq \left(\bigcup_C C(\ZZ,B) \right) \cup \left( \bigcup_\yy \{\yy\}\right),
\end{equation}
where $C$ runs over a collection of
$$O(B^{16/(3\sqrt{3k})+1/\sqrt{k}+\epsi})$$
irreducible curves of degree at least $(k+3)/6$, and $\yy$ runs over a collection of
$$O(B^{16/(3\sqrt{3k})+2/\sqrt{k}+\epsi})$$
points.

To obtain the estimate \eqref{eq:main}, we now apply Pila's estimate \cite{Pila95}. If $C \subset \AA^4$ is an irreducible curve of degree $d$, then we have
\begin{equation}
\label{eq:pila}
\cN(C,B) \ll_{d,\epsi} B^{1/d+\epsi}.
\end{equation}
Thus we conclude that
\begin{equation*}
\label{eq:mainestimate}
\cR_0(N,B) \ll_\epsi B^{16/(3\sqrt{3k}) + 1/\sqrt{k} + 6/(k+3)+ \epsi} + B^{16/(3\sqrt{3k}) + 2/\sqrt{k} + \epsi},
\end{equation*}
which establishes the main estimate in Theorem \ref{thm:main}.

It remains to estimate the number of special solutions, using known bounds for Thue equations.

\begin{prop}
\label{prop:thue}
Let $a,b,h \in \ZZ \setminus \{0\}$, and let $k \geq 3$ be an integer. Then the number of integer solutions $(x,y)$ to the equation $a x^k + b y^k = h$ is $O(h^\epsi)$ for any $\epsi >0$, where the implied constant depends only on $k$ and $\epsi$.
\end{prop}

\begin{proof}
More precisely, the number of solutions is at most $C^{1+\omega(h)}$, where $C$ is a constant depending only on $k$. This follows from Evertse's estimate \cite[Cor. 2]{Evertse84} for Thue-Mahler equations. Thus the proposition follows from the observation that $\omega(h) \ll \log h/\log\log h$.
\end{proof}

% \begin{proof}
% In case the polynomial $F(x,y)=a x^k + b y^k$ is irreducible over $\QQ$, this is a Thue equation. There is then a constant $C$ depending only on $F$, such that the number of solutions is at most $C^{\omega(h)+1}$, where $\omega(h)$ is the number of distinct prime factors of $h$. The result now follows from the observation that $\omega(h) \ll \log h/\log\log h$.

% Suppose, however, that $F$ is reducible. Then we can write $F(x,y) = G(x,y) H(x,y)$, where $F,G \in \ZZ[x,y]$ are relatively prime. For each integer divisor $d$ of $h$ there are at most $O_k(1)$ integer solutions to the system of equations
% \[
% G(x,y) = d, \quad H(x,y) = h/d
% \]
% by Bézout's theorem. The desired estimate now follows from the well-known fact \cite[Thm. 315]{Hardy-Wright} that the number of divisors of $h$ is $O(h^\epsi)$.
% \end{proof}

This result immediately implies the trivial bound 
\[
\cR(N,B) =O(B^{2+\epsi}).
\]
We shall now estimate $\cN(V_i,B)$ and $\cN(W_{i,j},B)$, where clearly it suffices to handle the case $(i,j)=(1,2)$.
Beginning with $\cN(V_1,B)$, we have at most two choices for the value of $x_1$. The number of integer triples $(x_2,x_3,x_4)$ satisfying
\begin{equation}
\label{eq:V_1}
a_2 x_2^k + a_3 x_3^k + a_4 x_4^k = 0,\quad -B \leq x_2,x_3,x_4 \leq B
\end{equation}
is $O_k(B)$. Indeed, choose $\epsi > 0$ so that $\theta := 2/k + \epsi < 1$. Then the number of \emph{primitive} solutions to \eqref{eq:V_1} is $O_k(B^\theta)$ by Heath-Brown's estimate \cite[Thm. 3]{Heath-Brown02}. Employing Möbius inversion, as in \cite[Ex. F.16]{Hindry-Silverman}, one sees that the total number of solutions is $O_k(B)$. Thus we conclude that $\cN(V_1,B) = O_k(B)$.

Next we consider $\cN(W_{1,2},B)$. Here we have $O(B)$ choices for $(x_3,x_4)$, and by Proposition \ref{prop:thue} there are $O_{k,\epsi}(N^\epsi)$ possibilities for $(x_1,x_2)$, so $\cN(W_{1,2},B) =O_{k,\epsi}(B N^\epsi)$. Thus we have established the following result. 
\begin{prop}
\label{prop:special}
Let $\cR_1(N,B)$ be the number of special solutions to \eqref{eq:fourpowers}
%$a_1 x_1^k + \dotsb + a_4 x_4^k = N$ 
satisfying $\max_i |x_i| \leq B$. Then we have
\[
\cR_1(N,B) = O_{k,\epsi} (B N^\epsi).
\]
\end{prop}

\section{The sum of three $k$-th powers and an $\ell$-th power}
\label{sec:wisdom}

We begin by proving Theorem \ref{thm:threepowers}. Let $Y \subset \AA^3$ be the closed subvariety defined by the equation
\[
a_1 x_1^k + a_2 x_2^k + a_3 x_3^k = M,
\]
and $Y_0 \subset Y$ the open subset defined by $a_i x_i^k \neq M$ for $i=1,2,3$. The estimate we seek to establish is then
\[
\cN(Y_0,B) = O_{k,\epsi}(B^{2/\sqrt{k}+\epsi}).
\]

By Theorem \ref{thm:Salberger7.2} there is a collection $\mathcal C$ of (geometrically integral) curves $C \subset Y$ of degree $O_{k}(1)$, such that all but $O_{k,\epsi}(B^{2/\sqrt{k} + \epsi})$ points in $Y(\ZZ,B)$ belong to one of the curves $C \in \mathcal C$, where $\#\cC = O_{k,\epsi}(B^{1/\sqrt{k}+\epsi})$. In other words, we have
\begin{equation}
\label{eq:curves+points}
\cN(Y_0,B) \leq \sum_{C \in \cC} \cN(C\cap Y_0,B) + O_{k,\epsi}(B^{2/\sqrt{k}+\epsi}).
\end{equation}

Let $\bar Y \subset \PP^3$ be the projective closure of $Y$, that is the Fermat hypersurface given by the equation
\[
-M x_0^k + a_1 x_1^k + a_2 x_2^k + a_3 x_3^k = 0.
\]
Since $\bar Y$ is smooth, it follows from a theorem of Colliot-Thélène \cite{Colliot-Thelene02} that the number of geometrically integral curves on $\bar Y$ that have degree at most $k-2$ is $O_k(1)$.
% Thus only $O_k(1)$ of the curves $C \in \cC$ have degree less than or equal to $k-2$.

Using the results of Salberger in Section \ref{sec:fermat}, we can say more about the degrees of these curves. Indeed, unless $C\in \cC$ is one of the standard lines, in which case $C \cap Y_0 = \varnothing$, it has degree at least $(k+1)/3$.

Again we use Pila's estimate \eqref{eq:pila}. The bounded number of curves $C \in \cC$ with $\deg C \leq k-2$ thus contribute $O_{k,\epsi}(B^{3/(k+1) + \epsi})$ to \eqref{eq:curves+points}, while the curves with degree at least $k-1$ contribute $O_{k,\epsi}(B^{1/\sqrt{k} + 1/(k-1) + \epsi})$. In sum, as $k \geq 3$, we get
\begin{align*}
\cN(Y_0,B) & \ll_{k,\epsi} B^{2/\sqrt{k}+\epsi} + B^{1/\sqrt{k} + 1/(k-1) + \epsi} + B^{3/(k+1) + \epsi} \ll B^{2/\sqrt{k}+\epsi},
\end{align*}
as desired. 

Now we turn to the proof of Theorem \ref{thm:kl}. Let $X \subset \AA^4$ be the hypersurface defined by the equation \eqref{eq:lth-and-kthpowers}. We shall count integral points on hyperplane sections of $X$. Thus, for each integer $a \in [0,N^{1/\ell})$, let $X_a$ be the intersection of $X$ with the hyperplane given by $x_4 = a$. Viewed as a subvariety of $\AA^3$, $X_a$ is given by the equation
\begin{equation}
\label{eq:X_a}
x_1^k +x_2^k+x_3^k = N-a^\ell.
\end{equation}
Let $B = N^{1/k}$. It is then obvious that we have
\begin{equation}
\label{eq:slicing}
R_{k,\ell}(N) \leq \sum_{0 \leq a < N^{1/\ell}} \cN_+(X_a,B) + 1.
\end{equation}
%Now we use Theorem \ref{thm:threepowers} with $a_i = 1$ and $M = N - a^\ell$.
As we are now only considering non-negative solutions to \eqref{eq:X_a}, Theorem \ref{thm:threepowers} implies that
\[
\cN_+(X_a,B) = O_{k,\epsi}(B^{2/\sqrt{k}+\epsi}) = O_{k,\epsi'}(N^{2/k^{3/2}+\epsi'}).
\]
Inserting this into \eqref{eq:slicing}, we get
\[
R_{k,\ell}(N) \ll_{k,\epsi} N^{1/\ell + 2/k^{3/2} + \epsi},
\]
which proves Theorem \ref{thm:kl}.

\bibliographystyle{plain}
\bibliography{ratpoints}

\end{document}